\documentclass[12pt,english,russian]{article}
\usepackage[T2A]{fontenc}
\usepackage[utf8]{inputenc}
\usepackage[dvipsnames]{xcolor}
\usepackage{amssymb,amsfonts,amsmath,mathtext,amsthm}
\usepackage{cite,enumerate,float,indentfirst}
\usepackage{graphics,epsfig,graphicx,epstopdf,amscd,latexsym,verbatim,hyperref,ulem}
\usepackage[top = 30mm, bottom = 20mm, left = 30mm, right = 25mm]{geometry}
\usepackage{centernot}

\usepackage{pgfplots,comment}
\pgfplotsset{compat=1.9}
\usepackage{filecontents}
\usetikzlibrary {arrows.meta}
\usepackage{blindtext, rotating}

\newtheorem{lemma}{Lemma}[section]

\newtheorem{theorem}{Theorem}[section]
\newtheorem{corollary}{Corollary}[section]

\newtheorem{definition}{Definition}[section]
\newtheorem{remark}{Remark}[section]

\def\Imm{\mathrm{Im}\,}

\def\Span{\mathrm{Span}}

\begin{document}
\begin{center}
{\Large Monomial Rota--Baxter and averaging operators on $F[x, x^{-1}]$}

\smallskip

Artem Khodzitskii 
\end{center}

\begin{abstract}
The study of Rota--Baxter operators on the polynomial algebra in one variable 
was initiated by S.H.~Zheng, L.~Guo, and M.~Rosenkranz (2015).
Monomial Rota--Baxter operators on $F[x]$ were classified in 2016,
while monomial averaging operators on $F[x]$ and $F_0[x]$ were classified in 2024.
In the current work, we continue the study of these operators on Laurent polynomial algebra $F[x, x^{-1}]$
over a field of arbitrary characteristic.

\medskip
{\it Keywords}:
Rota--Baxter operator, averaging operator, polynomial algebra, Laurent polynomial.

MSC code: 16W99
\end{abstract}

\section{Introduction}

Integral and differential operators are extensively employed in applied 
and theoretical mathematics, physics, and other scientific disciplines. 
As a result, these operators naturally attract significant interest. 
In turn, the Rota--Baxter operator is an algebraic generalization of the integral operator.
The Rota--Baxter operators were first introduced by G.~Baxter in 1960~\cite{Baxter}. 
However, earlier works by F.\,G.~Tricomi~\cite{Tricomi} and M.~Cotlar~\cite{Cotlar}
had defined relations similar to Baxter's definition.

\begin{definition}
A linear operator $R$ on an algebra $A$ defined over a field $F$ 
is called a~Rota--Baxter operator (an RB-operator, for short)
if the following relation
\begin{gather}\label{RBO}
R(a)R(b) = R\big(R(a)b + a R(b) + \lambda a b\big)
\end{gather}
holds for all $a, b \in A$.
Here $\lambda\in F$ is a fixed scalar called a weight of $R$.
\end{definition}

When $\lambda = 0$, relation~\eqref{RBO} is a generalization of 
the integration by parts formula.

G.-C.~Rota contributed significantly to the study and popularization of these operators 
from the 1960s to the 1990s~\cite{Rota}.
Rota--Baxter operators are closely related to various forms of the Yang--Baxter equation 
and have found applications in quantum field theory 
as well as various combinatorial formulas~\cite{GuoMonograph}. 
In the 1980s, they were independently rediscovered in mathematical physics and appear 
in connection with classical and modified versions of the Yang--Baxter equation~\cite{BelaDrin82,Semenov83}.
Today, the Rota--Baxter operators actively studied 
on a wide range of algebraic structures, including algebras and groups.

Another generalization of the integral operator is the averaging operator. 
An averaging operator is an idempotent Reynolds operator.
These operators were introduced by Reynolds~\cite{Reynolds} 
and have been actively studied ever since. 

\begin{definition}
A linear operator $T$ on an algebra $A$ defined over a field $F$
is called an averaging operator if the following relations hold for all $a,b \in A$:
\begin{gather}\label{Averaging}
T(a)T(b) = T(T(a)b) = T(a T(b)).
\end{gather}
\end{definition}

A significant contribution to the theory of averaging operators 
was made by J.~Kamp\'{e} de F\'{e}riet, 
who studied them over a period of about 30 years~\cite{KampeDeFeriet}.

Averaging operators have been studied in the context of functional analysis, 
with most work has focused on specific algebraic structures 
such as function spaces and Banach algebras. 
A connection between integration theory and averaging operators 
in the fields of turbulence theory and probability has been established. 
Through the study of Rota--Baxter operators, 
the algebraic study of averaging operators has been further developed and generalized. 
A more comprehensive historical overview can be found in~\cite{PeiGuoAveraging}. 

The study of Rota--Baxter operators on the polynomial algebra in one variable 
was initiated by S.H.~Zheng, L.~Guo, and M.~Rosenkranz~\cite{Monom2} in 2015.
An operator is called monomial if it maps each monomial to 
a monomial multiplied by a coefficient from the field.
Injective monomial RB-operators of weight zero on $F[x]$ were described in~\cite{Monom2}.
The description of monomial RB-operators of an arbitrary weight $\lambda$
without the injectivity condition was presented in~\cite{Monom}
and on the non-unital algebra $F_0[x]$ appeared in~\cite{MonomNonunital}.
In~\cite{GubPer}, the hypothesis of S.\,H. Zheng, L.~Guo, and M.~Rosenkranz
was confirmed by V.~Gubarev and A.~Perepechko regarding injective 
and not necessarily monomial Rota--Baxter operators on $F[x]$, 
where $F$ is a field of characteristic zero.
In~\cite{ReprRB,ReprRB2,ReprRB3}, finite-dimensional irreducible representations of 
certain Rota--Baxter algebras defined on $F[x]$ were described.

In our works~\cite{Khodzitskii, Khodzitskii2, Khodzitskii3}, 
we continued the study of monomial Rota--Baxter operators 
on polynomial algebras in several variables, 
with special attention to the two-variable case.
In~\cite{Khodzitskii2}, various classes of monomial Rota--Baxter 
and averaging operators on $F[x,y]$ and $F_0[x,y]$ were studied. 
In~\cite{Khodzitskii} and~\cite{Khodzitskii3}, 
the author describes monomial RB-operators of arbitrary weight $\lambda$ 
coming from averaging operators on algebras $F[x, y]$ and $F_0[x, y]$. 

In the current work, we consider monomial averaging and Rota--Baxter operators
on the Laurent polynomial algebra $F[x, x^{-1}]$.
We also treat the case of a field of nonzero characteristic.
It is natural to consider Laurent polynomials.
For instance, works on integro-differential rings~\cite{RaabRegensburger, DuRaab, RaabRegensburger2}
involve both polynomial algebras and Laurent polynomial algebras. 

The work is organized as follows. 

In \S2, we provide necessary preliminaries.

In \S3, we show that there exist only zero Rota--Baxter operators 
on commutative alternative algebras over a field of nonzero characteristic.
We also show that the same conclusion holds for the algebra $F[X, X^{-1}]$, 
where $X$ is an arbitrary set of variables and~$F$ is a field of characteristic zero.
In the case of nonzero weight, we describe monomial Rota--Baxter operators
on $F[x, x^{-1}]$ over a field of arbitrary characteristic.

In \S4, we describe monomial averaging operators on $F[x, x^{-1}]$ 
over a field of arbitrary characteristic. 

\section{Preliminaries}
 
Let $F$ be a field. We use the notation $F^* = F \setminus \{0\}$.
Unless otherwise stated, we assume that $F$ is a field of arbitrary characteristic.

Let $I$ be a nonempty set of indices and $X = \{ x_i \mid i \in I\}$.
Set $X^{-1} = \{ x^{-1}_i \mid i \in I\}$.

\begin{definition}
Let $X = \{ x_i \mid i \in I\}$, $I \neq \emptyset$.
The set of monomials in the Laurent polynomial algebra $F[X, X^{-1}]$ is defined as
$$
M(X) = 
\left\{ 
\prod_{i \in I} x_i^{\beta_i} \mid 
x_{\alpha} \in X, \,\, \beta_i \in \mathbb{Z}
\right\}.
$$
\end{definition}

\begin{lemma} [\!\!\cite{Cao}] \label{closure_AvOp_lemma}
Let $T$ be an averaging operator on an algebra $A$.
Then the following statements hold:

a) $\Imm T$ is a subalgebra in $A$.

b) $\Imm T\cdot\ker T$, $\ker T\cdot\Imm T\subseteq \ker T$.
\end{lemma}

\begin{lemma} [\!\!\cite{GuoMonograph,MonomNonunital}] \label{splitting_RB_lemma}
Let an algebra $A$ splits as a vector space into the direct sum of two
subalgebras $A_1$ and $A_2$. Then an operator $P$ defined by $P(a_1 + a_2) = -\lambda a_2$
for all $a_1 \in A_1$ and $a_2 \in A_2$
is an RB-operator of weight $\lambda$ on $A$.
\end{lemma}

\begin{definition}
An operator $P$ defined in Lemma~\ref{splitting_RB_lemma}
is called a splitting Rota--Baxter operator.
\end{definition}

\begin{lemma} [\!\!\cite{GuoMonograph,BGP,MonomNonunital}]\label{1InF_RB_lemma}
Let $R$ be a Rota--Baxter operator of weight $\lambda$ on a unital algebra~$A$
such that $R(1) \in F$.
Then the following statements hold:

a) If $\lambda \neq 0$, then $R(1) \in \{ 0, -\lambda\}$ 
and $R$ is a splitting Rota--Baxter operator,
where $A = \Imm R \oplus \ker R$ and $R^2 + R = 0$.

b) If $\lambda = 0$, then $R(1) = 0$.
\end{lemma}

\begin{lemma} \label{RB_mainLemma_F[x,x^-1]}
Let $X = \{ x_i \mid i \in I\}$, $I \neq \emptyset$.
If $R$ is a monomial RB-operator of nonzero weight on $F[X, X^{-1}]$, 
then $R(1) \in \{ 0, -\lambda\}$.
\end{lemma}

\begin{proof}
We prove a) by contradiction. 
Suppose that $R(1) = \alpha w$, where $\alpha \neq 0$ and $w \in M(X) \setminus F$.
Then, by~\eqref{RBO}, we obtain
$$
R(1)R(1) = 
\alpha^2 w^2 = 
R(2R(1) + 1) =
2\alpha R(w) + \alpha w.
$$
This contradicts the monomiality of $R$ for $w$, hence $R(1) \in \{ 0, -\lambda\}$. 

Substituting $z \in M(X)$ and $1$ into~\eqref{RBO} gives
\begin{gather*}
R(z)R(1) = 0 = R(R(z) + z), 
\,\, \mbox{if}\,\,  R(1) = 0, \\
R(z)R(1) = -\lambda R(z) = R(R(z) - \lambda z + \lambda z), 
\,\, \mbox{if} \,\,  R(1) = -\lambda.
\end{gather*}
In the both cases $R(R(z)) = - R(z)$.
\end{proof}

\begin{lemma}[\!\!\cite{GuoMonograph,BGP,MonomNonunital}] \label{RB_under_automorphism}
Let $A$ be an algebra and let $P$ be an RB-operator of weight $\lambda$ on~$A$.
Then the following statements hold:

a) The operator $\alpha^{-1} R$, $\alpha \in F^*$,
is an RB-operator of weight $\alpha\lambda$ on $A$.

b) For any automorphism $\psi \in \mathrm{Aut}(A)$, the operator $\psi^{-1} P \psi$
is an RB-operator of weight $\lambda$ on $A$.
\end{lemma}

Note that we denote by $F_0[x]$ the ideal of $F[x]$ generated by $x$.

\begin{theorem}[\!\!\cite{Monom,MonomNonunital}]\label{RB_on_F[x]_and_F0[x]}
let $R$ be a monomial RB-operator of weight zero on $F_0[x]$ ($F[x]$),
where $F$ is a field of characteristic zero. 
Then there exist $0 < m$, $p_i \in \mathbb{N}$, and $q_i \in F$, 
$0 < i \leqslant m$ ($0 \leqslant i < m$), such that
$p_i = 0$ if and only if $q_i = 0$,
and $R$ has the form
$$
R(x^{m a+b})= q_b \frac{x^{m(a+p_b)}}{m(a+p_b)},
$$
where $a \in \mathbb{N}$ and $0<b \leqslant m$ ($0 \leqslant b<m$).
\end{theorem}

\begin{theorem}[\!\!\cite{Monom,MonomNonunital}] \label{R_on_k_[x]_theorem}
Up to conjugation by automorphisms of $F[x]$, 
every Rota--Baxter operator of nonzero weight on $F[x]$
is splitting with subalgebras $F$ and $\Span (\{x\})$.
That~is, it corresponds to one of the following operators:
$$
R(x^n) = (-1)^{n + 1}, \quad
R(x^n) = 
\begin{cases}
-x^n , & n > 0, \\
\theta, & n = 0, 
\end{cases}
\,\, \theta \in \{ 0, -1\}.
$$
\end{theorem}

\section{Rota--Baxter operators}

In this section, we consider the classification of 
Rota--Baxter operators of arbitrary weight 
on Laurent polynomial algebras.

\subsection{Zero weight}

\begin{lemma} \label{R(z)=gamma}
Let $X = \{ x_i \mid i \in I\}$, $I \neq \emptyset$,
and let $R$ be an RB-operator of weight $1$ on $F[X, X^{-1}]$.
If $R(z) = \gamma \in F$ for some $z \in M(X)$,
then $R(z^n) = (-1)^{n + 1} \gamma^n$ for all $n \geqslant 1$.
\end{lemma}
\begin{proof}
Let $R(z) = \gamma \in F$. 
We prove the statement by induction on $n$.
The base case $n = 1$ is evident.
Let us prove the induction step for $n + 1$.
By~\eqref{RBO} we have
$$
R(z^n)R(z) = (-1)^{n + 2} \gamma^{n + 1} =
R((-1)^{n + 1} \gamma^n z + \gamma z^n + z^{n + 1}).
$$
Thus $R(z^{n + 1}) = -(-1)^{n + 2} \gamma^{n + 1}$, 
which completes the induction step.
\end{proof}

\begin{theorem} \label{theorem_monomial_zero_weight_on_F[X,X^{-1}]}
Let $X = \{ x_i \mid i \in I\}$, $I \neq \emptyset$.
If $R$ is a monomial Rota--Baxter operator of weight zero on $F[X, X^{-1}]$, then $R = 0$.
\end{theorem}
\begin{proof}
Let $R$ be a monomial RB-operator of weight zero on $F[X, X^{-1}]$ and let $J \subseteq I$.
Suppose there exists a monomial $\prod_{i \in J} x^{a_i}_i$, $a_i \in \mathbb{Z}$, 
such that
$
R\left(\prod_{i \in J} x^{a_i}_i\right) = \alpha \prod_{i \in J} x^{b_i}_i,
$
$\alpha \in F^*$.
We also have
$
R\left(\prod_{i \in J} x^{a_i - b_i}_i\right) = \beta \prod_{i \in J} x^{c_i}_i,
$
$\beta \in F$, and by~\eqref{RBO} we obtain the following:
\begin{multline*}
R\Bigl(\prod_{i \in J} x^{a_i}_i\Bigl) 
R\Bigl(\prod_{i \in J} x^{a_i - b_i}_i\Bigl) = 
\alpha \beta \prod_{i \in J} x^{b_i}_i \cdot 
\prod_{i \in J} x^{c_i}_i \\ =
R\Bigl(\alpha \prod_{i \in J} x^{b_i}_i \cdot 
\prod_{i \in J} x^{a_i - b_i}_i + 
\beta \prod_{i \in J} x^{a_i}_i 
\prod_{i \in J} x^{c_i}_i\Bigl) =
\alpha^2 \prod_{i \in J} x^{b_i}_i + 
\beta R\Bigl(\prod_{i \in J} x^{a_i + c_i}_i\Bigl).
\end{multline*}
If $\beta = 0$, then we get a contradiction with $\alpha \neq 0$.
If $\beta \neq 0$, then 
$
\prod_{i \in J} x^{b_i + c_i}_i =
\prod_{i \in J} x^{b_i}_i
$
and $c_i = 0$ for all $i \in J$.
Hence $\beta 1\in \Imm R$, which contradicts Lemma~\ref{1InF_RB_lemma}.
\end{proof}

\begin{theorem} \label{theorem_zero_weight_comm_alt}
Let $A$ be a commutative alternative algebra over a field $F$ of characteristic~$p > 0$.
If $R$ is a Rota--Baxter operator of weight zero on $A$, then $R = 0$.
\end{theorem}
\begin{proof}
It is a well-known fact that, according to Artin's theorem, 
an algebra~$A$ is an alternative if and only if,
for all $a, b \in A$ the subalgebra 
generated by $\{a,b\}$ is an associative algebra.

Assume $R \neq 0$. Then there exists $a \in A$
such that $R(a) = b$, $b \neq 0$.
By~\eqref{RBO} we obtain $R(a)R(a) = b^2 = 2 R(ab)$.
If $p = 2$, then we have a contradiction $0 \neq b^2 = 0$.
So, we may assume that $p \neq 2$. Therefore $R(ab) = b^2/2$.

By~\eqref{RBO} for $a$ and $ab$ we obtain the following:
$$
R(a)R(ab) 
 = b^3/2 
 = R( bab + ab^2/2) 
 = (3/2) R(ab^2).
$$

Similarly, we have found that $R(ab^{p - 2}) = (p - 1)^{-1} b^{p - 1}$.
We then reach a contradiction:
$$
0 \neq R(a)R(ab^{p - 2}) 
 = (p - 1)^{-1} b^p 
 = R( ab^{p - 1} + (p - 1)^{-1}  ab^{p - 1}) 
 = \frac{p}{p-1}R( ab^{p - 1})
 = 0.
$$
\end{proof}

Theorem~\ref{theorem_monomial_zero_weight_on_F[X,X^{-1}]} 
is significantly stronger than Theorem~\ref{theorem_zero_weight_comm_alt}, 
but it does not cover the case of a field of characteristic zero.

\begin{corollary}
Let $A$ be a commutative (associative) algebra over a field $F$ of characteristic $p > 0$.
If $R$ is a Rota--Baxter operator of weight zero on $A$, then $R = 0$.
\end{corollary}

In particular, there exists only trivial monomial 
RB-operator of weight zero on $F[x, x^{-1}]$.
The set of all monomial RB-operators on $F[x]$ 
is described in Theorem~\ref{RB_on_F[x]_and_F0[x]} and is nontrivial.

\subsection{Nonzero weight}

Let us introduce the automorphisms $\psi_\alpha$, $\alpha \in F^*$,
and $\psi_{x, x^{-1}} \in \mathrm{Aut} (F[x, x^{-1}])$
defined by $\psi_\alpha(x^{\pm 1}) = (x / \alpha)^{\pm 1}$ 
and $\psi_{x, x^{-1}}(x^{\pm 1}) = x^{\mp 1}$. 

\begin{theorem}
Let $R \neq 0$ be a monomial Rota--Baxter operator of weight $1$ on the algebra $F[x, x^{-1}]$.
Then, up to conjugation by the automorphisms $\psi_\alpha$ and $\psi_{x, x^{-1}}$,
the operator $R$ coincides with one of the following operators:
\begin{equation*} 
R_1(x^n) = (-1)^{n + 1}, \,\, 
R_2(x^n) = -x^n, \,\, 
n \in \mathbb{Z},
\quad
R_3(x^n) = 
\begin{cases}
-x^n, & n > 0, \\
(-1)^{n + 1}\gamma^n, & n \leqslant 0,
\end{cases}
\,\,
\gamma \in F.
\end{equation*}
\end{theorem}

\begin{proof}
Throughout the proof we will use the relation~\eqref{RBO} for $x$ and $x^{-1}$:
\begin{equation} \label{Loran_nonzero_1var_R(x)R(x^{-1})}
R(x)R(x^{-1}) = R(R(x)x^{-1} + xR(x^{-1})) + R(1).
\end{equation}

By Lemmas~\ref{1InF_RB_lemma} and~\ref{RB_mainLemma_F[x,x^-1]}
we have $R(1) = \theta \in \{0,-1\}$ and $R|_{\Imm R} = -\mathrm{id}$.

We split the proof into several cases according to the action of $R$ on
$x$ and $x^{-1}$. In Case~0 we prove an auxiliary statement.
Case~1 treats the situation where one of $x$, $x^{-1}$ lies in the kernel.
The main cases, where $R(x), R(x^{-1}) \neq 0$, are dealt with in
Cases~2 and~3.

{\sc Case 0}. $R(x), R(x^{-1}) \neq 0$.
We show that if there exists a monomial $x^m \in \ker R$ and $|m| > 1$, then $R=0$.
Assume $m > 0$ and let $n \geqslant 1$ be the smallest integer such that $R(x^n)\neq 0$,
hence $R(x^s) = 0$ for all $0 \leqslant s < n$.
Thus, by~\eqref{RBO} we have a contradiction:
$$
0 = R(x^{n - 1})R(x) = R(0 + 0 + x^n) \neq 0.
$$
The same arguments works for $m < 0$, up to conjugation by $\psi_{x,x^{-1}}$. 

{\sc Case 1}. $R(x) = 0$ or $R(x^{-1}) = 0$. 
If $x, x^{-1} \in \ker R$, then $R = 0$ by Lemma~\ref{1InF_RB_lemma}.
So, we may assume $R(x) = 0$ and $R(x^{-1}) = \beta x^b$,
where $\beta \neq 0$ and $b \in \mathbb{Z}$.

{\sc Case 1.1}. If $b = 0$, then by Lemma~\eqref{R(z)=gamma}
we have $R(x^{-n}) = (-1)^{n + 1} \beta^n$.
Equation~\eqref{Loran_nonzero_1var_R(x)R(x^{-1})} implies that $R(1) = 0$.
Applying~\eqref{RBO} to $x$ and $x^{-2}$ we obtain
$$
R(x)R(x^{-2}) = 0 =
R(0 - \beta^2 x + x^{-1}) = \beta,
$$
so $\beta = 0$, a contradiction. 

{\sc Case 1.2}. If $b > 0$, 
then $R(1) = 0$ by~\eqref{Loran_nonzero_1var_R(x)R(x^{-1})}.
We prove by induction that $R(x^{-n}) = \beta^n x^{nb}$ for all $0 < n \leqslant b + 1$. 
The base case $n=1$ is evident.
Assume the formula holds for all $s \leqslant k$, 
and prove it for $k + 1 \leqslant b + 1$.
Using~\eqref{RBO},
$$
R(x^{-k})R(x^{-1}) = \beta^{k + 1} x^{(k + 1)b} =
R(\beta^k x^{kb - 1} + \beta x^{b - k} + x^{- k - 1}).
$$
Since $kb - 1 \geqslant 0$ and $b - k \geqslant 0$,
$R(x^{-(k + 1)}) = \beta^{k + 1} x^{(k + 1)b}$.
In particular, $R(x^{-(b + 1)}) = \beta^{b + 2} x^{(b + 1)b}$.
Then, we get a contradiction to the monomiality of $R$ for $x^{-(b + 2)}$ by~\eqref{RBO}:
\begin{multline*}
R(x^{-(b + 1)})R(x^{-1}) = \beta^{b+2} x^{(b + 1)b} \\ =
R(\beta^{b + 1} x^{(b + 1)b - 1} + \beta x^{-b - 1 + b} + x^{-(b + 2)}) =
0 + \beta^2 x^b + R(x^{-(b + 2)}).
\end{multline*}

{\sc Case 1.3}. Let $b < 0$.
If $b = -1$, then, up to conjugation by $\psi_{x, x^{-1}}$, 
the operator $R$ coincides with $R_3$ with $\gamma = 0$.
Now assume $b < -1$ and derive a contradiction.
From the equation~\eqref{Loran_nonzero_1var_R(x)R(x^{-1})}
we obtain $R(\beta x^{b + 1} + 1) = 0$.
At the same time $R(1) \in \{0, -1\}$ and $b + 1 < 0$,
therefore $R(x^{b + 1}) \in F$ and by Lemma~\ref{R(z)=gamma},
$R(x^{a(b + 1)}) \in F$ for every $a>0$.
On the other hand, by Lemma~\ref{1InF_RB_lemma},
$R(x^{cb}) = -x^{cb} \not \in F$ for $c > 0$.
Choosing $a = b$ and $c = b + 1$ gives a contradiction.

The proof of Case~1 holds when $R(x) \neq 0$ and $R(x^{-1}) = 0$,
up to conjugation by $\psi_{x,x^{-1}}$. 

From Cases~0 and~1 we see that the only remaining possibility is
$R(x^n) \neq 0$ for every $|n| > 0$.
Let $R(x) = \alpha x^a$ and $R(x^{-1}) = \beta x^b$, 
where $a, b \in \mathbb{Z}$ and $\alpha, \beta \neq 0$.
We consider cases according to the values of $a$ and $b$.

{\sc Case 2}. Let $a=0$, then $R(x) = \alpha$.
By Lemma~\ref{R(z)=gamma}, $R(x^n) = (-1)^{n + 1} \alpha^n$, $n > 0$.
It is easy to see that $b \leqslant 0$.
Indeed, if $b>0$, then $-x^b = R(x^b) = (-1)^{b + 1} \alpha^b$,
which is impossible.

{\sc Case 2.1}. If $b = 0$,
then $R(x^{-n}) = (-1)^{n + 1} \beta^n$, $n > 0$, by Lemma~\ref{R(z)=gamma}.
Equation~\eqref{Loran_nonzero_1var_R(x)R(x^{-1})} implies $-\alpha \beta = \theta \in \{ 0, -1\}$.
Since $\alpha, \beta \neq 0$, it follows that
$\beta = \alpha^{-1}$ and $R(1) = -1$.
Hence, up to conjugation by $\psi_\alpha$, the operator $R$
coincides with $R_1$.

{\sc Case 2.2}.
If $b < 0$, then $R(x^{-b}) = (-1)^{-b + 1} \alpha^{-b}$
and by~\eqref{RBO} we have
$$
R(x^{-1})R(x^{-b}) = (-1)^{-b + 1} \alpha^{-b} \beta x^b =
R(\beta x^{b - b} + (-1)^{-b + 1} \alpha^{-b} x^{-1} + x^{-b - 1}).
$$
Hence, $R(x^{-b - 1}) = -\beta \theta$, 
where $\theta \in \{ 0, -1\}$.
Since $x^{-b - 1} \not \in \ker R$, we conclude that $\theta = -1$.

If $b = -1$, then $\beta = -1$, since $R(x^{-1}) = -x^{-1}$.
Moreover, $R(x^{-n}) = -x^{-n}$ by Lemma~\ref{1InF_RB_lemma}.
For any $n \geqslant 0$ and $m > 0$ the relation~\eqref{RBO} holds:
$$
R(x^n)R(x^{-m}) = (-1)^{n + 2} \alpha^n x^{-m} =
R((-1)^{n + 1} \alpha^n x^{-m} - x^{n - m} + x^{n - m}) =
(-1)^{n + 2} \alpha^n x^{-m}.
$$
Thus, $R$ is a Rota--Baxter operator and, up to conjugation by
$\psi_{x,x^{-1}}$, coincides with $R_3$. 

It remains to show that if $b < -1$, 
then $R$ is not a Rota--Baxter operator.
We have already shown that $R(x^{-b - 1}) = \beta$.
Now apply~\eqref{RBO} to $x^{-b-1}$ and $x^{-1}$:
$$
R(x^{-b - 1})R(x^{-1}) = \beta^2 x^b =
R(\beta x^{-1} + \beta x^{-b - 1}x^b + x^{-b - 2}) =
2\beta^2 x^b + R(x^{-b - 2}).
$$
From the inequality $b < -1$ follows that $-b \geqslant 2$,
then $R(x^{-b - 2}) = (-1)^{-b - 1} \alpha^{-b - 2}$
and we have a contradiction
$-\beta^2 x^b= R(x^{-b - 2}) = (-1)^{-b - 1} \alpha^{-b - 2}$.

{\sc Case 3}. If $a \neq 0$ and $b = 0$, 
then this case is similar to
Case~2.2, up to conjugation by~$\psi_{x,x^{-1}}$.
It remains to consider the case where both $a$ and $b$ are nonzero.

{\sc Case 3.1}.
Assume that there exist minimal integers $p, q > 0$,
such that $x^p, x^{-q} \in \Imm R$.
If $p \neq q$, then $p - q < p$ and $p - q > - q$,
which contradicts either the minimality of $p$ or that of $q$.
Hence $p=q$.
Applying~\eqref{RBO} to $x^p$ and $x^{-p}$, we obtain $R(1) = -1$:
$$
R(x^p)R(x^{-p}) = 1 = R(-1 - 1 + 1) = -R(1).
$$

If $p=1$, then $R$ coincides with $R_2$.

Assume that $p > 1$.
We show that $\Imm R \cap M(\{x\}) = \{ x^{sp} \mid s \in \mathbb{Z}\}$.
The inclusion $\{ x^{sp} \mid s \in \mathbb{Z}\} \subseteq \Imm R$
is immediate from Lemma~\ref{1InF_RB_lemma}.
Consider the reverse inclusion.
Suppose, to the contrary, that $x^{cp + b} \in \Imm R$,
$c \in \mathbb{Z}$, $0 \leqslant b < p$.
Then $x^{(c + s)p + b} \in \Imm R$ for any $s \in \mathbb{Z}$,
which contradicts the minimality of $p$ when $s = -c$.
If for some $n \in \mathbb{Z} \setminus \{0\}$ we have $R(x^n) \in F^*$,
then $R(x^{sn}) \in F^*$ for all $s > 0$ by Lemma~\ref{R(z)=gamma}.
This is impossible, since these conditions 
imply the contradictory equality $-x^{np} = R(x^{np}) \in F^*$.
Therefore, $a = up$ and $b = vp$ for some $u, v \in \mathbb{Z}$.
Equation~\eqref{Loran_nonzero_1var_R(x)R(x^{-1})} becomes
$$
\alpha \beta x^{(u + v)p} =
\alpha R(x^{up - 1}) + \beta R(x^{vp + 1}) - 1.
$$
Consequently, either $R(x^{up - 1}) \in F^*$ or $R(x^{vp + 1}) \in F^*$,
which is impossible by the arguments above.

{\sc Case 3.2}.
It remains to consider the case where there exists 
a minimal integer $p > 0$ such that either $x^p \in \Imm R$ or $x^{-p} \in \Imm R$.
Equivalently, $R(x^n) = \gamma_n x^{c_n}$ and for every $n \in \mathbb{Z}$ 
either $c_n \geqslant 0$ or $c_n \leqslant 0$, $\gamma_n \in F^*$.

Assume that $R(x^p) = -x^p$.
The case $p < 0$ is similar.
By the observation made earlier,
the subalgebra $F[x]$ is $R$-invariant.
So, we may apply Theorem~\ref{R_on_k_[x]_theorem}
and, according to the restrictions already established for $R$, we obtain
$$
R|_{F[x]} =
\begin{cases}
-x^n , & n > 0, \\
\theta, & n = 0,
\end{cases} 
\,\, \theta \in \{ 0, -1\}.
$$ 
Thus, $a = 1$ and $\alpha = \gamma_0 = -1$. 
Finally, substituting $x^{-1}$ and $x^{-1}$ into~\eqref{RBO},
we arrive at a contradiction to the monomiality of $R$ for $x^{-2}$:
$$
R(x^{-1})R(x^{-1}) = \beta^2 x^{2b} = 
2\beta R(x^{b - 1}) + R(x^{-2}),
$$
since $b > 0$, and therefore $R(x^{b - 1}) = -x^{b - 1}$.
\end{proof}

\begin{remark}
Note that the classification of monomial Rota--Baxter operators 
of nonzero weight on $F[x]$ and $F_0[x]$ continues to hold in the case $\mathrm{char}F \neq 2$.
The proof of Theorem~2 in~\cite{MonomNonunital} carries over without any changes.
However, in the case when $\mathrm{char}F = 2$, the classification becomes more complex.
In this case, the proof of the theorem no longer applies, 
as the inductive argument fails in the induction step of Case~2B.
\end{remark}

\section{Averaging operators}

\begin{theorem}\label{Averaging_on_F[x,x^-1]}
Let $T$ be a monomial averaging operator on $F[x, x^{-1}]$.
Then up to conjugation $\psi_{x, x^{-1}}$ either there exist $0 < m$, $p_i \in \mathbb{Z}$ and $q_i \in F$, 
$0 \leqslant i < m$, such that $T$ has the form
$$
T_1(x^{m a + b})= q_b x^{m(a + p_b)}, \,\, a \in \mathbb{Z}, \,\, 0 \leqslant b < m,
$$
or there exist $q_n \in F$, $n \in \mathbb{Z}$, 
such that $T$ has the form $T_2(x^n) = q_n$.
\end{theorem}
\begin{proof}
By monomiality, any averaging operator on $F[x, x^{-1}]$
is defined by the rule $T(x^n) = \alpha_n x^{t_n}$,
$\alpha_n \in F$, $t_n \in \mathbb{Z}$, for all $n \in \mathbb{Z}$.
It is evident that if $\Imm T \subseteq F$, then $T = T_2$.
So, we assume that $\Imm T \not\subseteq F$.

It is straightforward to verify that 
the following formula holds for any $z \in M(\{x\})$:
\begin{equation}\label{xn^k+1=xn^k_and_xn}
(T(z))^{k + 1} = T((T(z))^k z), \quad k \geqslant 0.
\end{equation} 

There exists $s \in \mathbb{Z}$ such that $T(x^s) \not \in F$, 
since $\Imm T \not\subseteq F$.
Set $m = \gcd (\{s \mid x^s \in \Imm R\})$, we may assume that $m > 0$.
If $x^a \in \Imm T$, then $m \mid a$ and by Lemma~\ref{closure_AvOp_lemma}
there exists $N \geqslant 0$ such that 
$\Span (\{x^{m k} \mid |k| \geqslant N \}) \subseteq \Imm T$.
The case $M = \Imm T \cap \{ x^k \mid k < 0\} = \emptyset$ 
is similar to the case $\Imm T \cap \{ x^k \mid k > 0\} = \emptyset$
up to conjugation $\psi_{x, x^{-1}}$.
Thus, we have two cases:
in Case~1, every $x^a \in \Imm T$ satisfies $a \geqslant 0$,
in Case~2, $a \in \mathbb{Z}$ in case 2.

We show that if $T(x^{ma + b}) = 0$ for any $a \in \mathbb{Z}$ and $0 \leqslant b < m$,
then $T(x^{mc + b}) = 0$ for all $c \in \mathbb{Z}$.
By the choice of $N$ and Lemma~\ref{closure_AvOp_lemma},
the subspace $\ker T$ is an $\Imm T$-bimodule.
Hence, $T(x^{mc + b}) = 0$ for all $c \geqslant a + N$ in Case~1
and for all $|c| \geqslant |a| + N$ in Case~2. 

Consider $c < a + N$ in Case~1 and $|c| < |a| + N$ in Case~2.
Suppose, to the contrary, that $T(x^{mc + b}) = \gamma x^{md}$, 
where $\gamma \neq 0$ and $d \in \mathbb{Z} \setminus \{ 0\}$.
By~\eqref{xn^k+1=xn^k_and_xn} we obtain 
\begin{equation}\label{Averaging_on_F[x,x^-1]_case_ker_eq}
\gamma^{k + 1} x^{md(k + 1)}  = (T(x^{mc + b}))^{k + 1} = 
T((T(x^{mc + b}))^k x^{mc + b}) = T(\gamma^k x^{m(dk + c) + b}).
\end{equation}
In Case~1, we have $a \geqslant 0$ and $d > 0$.
We may choose $k \geqslant 0$ such that $dk + c \geqslant N + a$.
Thus, in~\eqref{Averaging_on_F[x,x^-1]_case_ker_eq} we have
$0 \neq \gamma^{k + 1} x^{md(k + 1)} = 0$, a contradiction.
In Case~2 we also have a contradiction,
since we may choose $k \geqslant 0$ 
such that $dk + c \geqslant N + |a|$ if $d > 0$,
or $dk + c \leqslant - N - |a|$ if $d < 0$.

It remains to consider the case $d = 0$,
i.\,e. $T(x^{mc + b}) = \gamma \in F^*$.
For every integer $|l| \geqslant N$, 
there exists $d \in \mathbb{Z}$ such that $T(x^d) = \varepsilon x^{ml}$ and $\varepsilon \neq 0$.
Substituting $x^d$ and $x^{mc + b}$ into~\eqref{Averaging},
we obtain
$$
0 \neq
\varepsilon \gamma x^{ml} =
\gamma T(x^d) =
T(x^d T(x^{mc + b})) =
T(T(x^d)x^{mc + b}) =
\varepsilon T(x^{m(c + l) + b}).
$$
In both Cases~1 and~2, we have $T(x^{m(c + l) + b}) = 0$
when $|c + l| \geqslant |a| + N$.
It is obvious, that we may always choose $l$ 
such that this inequality holds,
which yields a contradiction.

Now let $0 \leqslant b < m$ be such that $T(x^b) = q_b x^{mt_b}$, 
where $q_b \neq 0$ and $t_b \in \mathbb{Z}$.
We prove that $T(x^{mc + b}) = q_b x^{m(c + t_b)}$ for all $c \in \mathbb{Z}$.
Consider $c \geqslant N$.
There exists $d \in \mathbb{Z}$ such that $T(x^d) = \gamma x^{mc}$ 
with $\gamma \neq 0$, due to $x^{mc} \in \Imm T$.
Applying~\eqref{Averaging}, we obtain
$$
q_b\gamma x^{m(c + t_b)} = T(x^b)T(x^d) = T(x^b T(x^d)) = \gamma T(x^{mc + b}).
$$
Now assume $c < N$ and $c \neq 0$.
Then $T(x^{mc + b}) = q_{mc + b} x^{m r}$, where $r \in \mathbb{Z}$
and $q_{mc + b} \neq 0$, since otherwise $T(x^b) = 0$.
The following expressions coincide, by~\eqref{Averaging}:
\begin{gather} \label{Averaging_on_F[x]_and_F0[x]_last_eq}
\begin{gathered}
T(x^b)T(x^{mc + b}) = q_bq_{mc + b} x^{m(t_b + r)}, \quad 
T(T(x^b)x^{mc + b}) = q_b T(x^{m(c + t_b) + b}), \\
T(x^b T(x^{mc + b})) = q_{mc + b} T(x^{mr + b}). 
\end{gathered}
\end{gather}
Thus, $T(x^{m(c + t_b) + b}) = q_{mc + b} x^{m(t_b + r)}$ 
and $T(x^{mr + b}) = q_b x^{m(t_b + r)}$.
If $r = c + t_b$, then the desired formula follows immediately.
Assume that $r \neq c + t_b$.
By~\eqref{Averaging_on_F[x]_and_F0[x]_last_eq}, 
$q_bx^{m(c + t_b) + b} - q_{mc + b}x^{mr + b} \in \ker T$.
Using Lemma~\ref{closure_AvOp_lemma} after multiplying 
this polynomial by $x^{mL} \in \Imm T$, $L \geqslant N$,
we again obtain an element of $\ker T$.
Choose $L$ such that $c + t_b + L \geqslant N$ and $r + L \geqslant N$ both hold.
Therefore, we obtain the equation
$$
0 = T(q_bx^{m(c + t_b + L) + b} - q_{mc + b}x^{m(r + L) + b}) 
 = q_b^2 x^{m(c + 2t_b + L)} - q_{mc + b}q_b x^{m(r + t_b + L)}.
$$
Since $q_b\neq 0$ and $q_{mc + b} \neq 0$, 
it follows that $m(c + 2t_b + L) = m(r + t_b + L)$,
a contradiction.

We show that only Case~2 can occur, and the operator $T$ is completely determined by
its values on the monomials $x^b$, $0 \leqslant b < m$.
It is easy to see that $T$ coincides with the operator~$T_1$.
\end{proof}

Comparing the classifications of monomial averaging operators on $F_0[x]$ and $F[x]$ (see~\cite{Khodzitskii2}) 
and $F[x, x^{-1}]$, it can be observed that they are defined by the same formulas.

\section*{Acknowledgements}

The author express his gratitude to his supervisor, Vsevolod Gubarev.

The study was supported by a grant from the Russian Science Foundation \linebreak
\textnumero 25-41-00005

\end{document}